\documentclass[11pt]{amsart}
\usepackage{amsmath,amssymb,amsthm,mathtools}
\usepackage[margin=1.15in]{geometry}
\usepackage{enumitem}
\usepackage{booktabs}
\usepackage{listings}
\usepackage{xcolor}
\usepackage[expansion=false]{microtype}
\usepackage[colorlinks=true,linkcolor=blue!50!black,citecolor=green!40!black,urlcolor=blue!50!black]{hyperref}

\newcommand{\C}{\mathbb{C}}
\newcommand{\Z}{\mathbb{Z}}
\newcommand{\Q}{\mathbb{Q}}
\newcommand{\F}{\mathbb{F}}
\newcommand{\PP}{\mathbb{P}}
\newcommand{\Der}{\operatorname{Der}}
\newcommand{\GL}{\operatorname{GL}}
\newcommand{\SL}{\operatorname{SL}}
\newcommand{\SU}{\operatorname{SU}}
\newcommand{\U}{\operatorname{U}}
\newcommand{\Gr}{\operatorname{Gr}}
\newcommand{\Conf}{\operatorname{Conf}}
\newcommand{\Poin}{\operatorname{Poin}}
\newcommand{\rk}{\operatorname{rk}}
\newcommand{\pd}{\operatorname{pd}}
\newcommand{\coker}{\operatorname{coker}}
\newcommand{\sgn}{\operatorname{sgn}}
\newcommand{\Mo}{\mathcal{M}_{0}}
\newcommand{\cA}{\mathcal{A}}
\newcommand{\cL}{\mathcal{L}}
\newcommand{\dd}[1]{\partial_{#1}}

\theoremstyle{plain}
\newtheorem{theorem}{Theorem}[section]
\newtheorem{proposition}[theorem]{Proposition}
\newtheorem{lemma}[theorem]{Lemma}
\newtheorem{corollary}[theorem]{Corollary}
\newtheorem{conjecture}[theorem]{Conjecture}
\newtheorem{mainthm}{Theorem}

\theoremstyle{definition}
\newtheorem{definition}[theorem]{Definition}
\newtheorem{example}[theorem]{Example}

\theoremstyle{remark}
\newtheorem{remark}[theorem]{Remark}
\numberwithin{equation}{section}
\begin{document}

\title[Geometry of some determinantal arrangements]{Geometry of some determinantal arrangements}

\author{Marek Janasz}
\address{Department of Mathematics, University of the National Education Commission\\
Podchorażych 2, 30-084 Kraków, Poland}
\email{marek.janasz@uken.krakow.pl}

\date{Draft, \today}


\begin{abstract}
Let $D_{m,n}\subset\C^{m\times n}$ be the union of the zero loci of the
maximal minors of a generic $m\times n$ matrix, and let $U_{m,n}$ be its
complement. We show that $U_{m,n}\cong\GL_m(\C)\times\Gr(m,n)^\circ$, where
$\Gr(m,n)^\circ$ is the open stratum of the Grassmannian on which all
Pl\"ucker coordinates are nonzero. For $m\in\{2,n-2\}$ this gives closed
formulas for the topology of $U_{m,n}$ and recovers results of Yim. Our main
result, which is computer-assisted, is that $D_{3,5}$ is \emph{not} a free
divisor. This answers a question of Yim. We determine $\Der(-\log D_{3,5})$
explicitly. It has projective dimension one, and it fails to be free exactly
on the matrices of rank at most two. There the failure comes from a non-free
transversal arrangement: the ten planes spanned by pairs of five general
vectors in $\C^3$. Computations in further cases support the conjecture that
$D_{m,n}$ is free only for $m\le 2$ or $n=m+1$.
\end{abstract}

\maketitle
\setcounter{tocdepth}{1}

\section{Introduction}

Let $X$ be a complex manifold and $D\subset X$ a reduced divisor. The sheaf $\Der_X(-\log D)$ of logarithmic vector fields consists of the vector fields tangent to $D$. We say that $D$ is \emph{free} if the sheaf $\Der_X(-\log D)$ is locally free \cite{Saito}. Free divisors appear as discriminants of versal deformations, as many classical hyperplane arrangements (Coxeter arrangements, supersolvable arrangements, graphic arrangements of chordal graphs \cite{Stanley,OT}), and as linear free divisors coming from prehomogeneous representations \cite{BM,GMNS}. For hyperplane arrangements freeness has strong topological consequences. For example, by Terao's factorization theorem \cite{Terao81} we know that the characteristic polynomial of a free arrangement, equivalently the Poincar\'e polynomial of its complement, factors into linear factors given by the degrees of a basis of the $\mathcal{O}_{X}$-module $\Der(-\log D)$.

This paper is concerned with the so-called \emph{determinantal arrangements of maximal minors}. Let $M=(x_{rk})$ be the generic $m\times n$ matrix with $m<n$ having columns $v_1,\dots,v_n\in\C^m$, and for an $m$-element subset $I\subset[n]$ let $\Delta_I$ be the maximal minor on the columns $I$. Put
\[f_{m,n}=\prod_{I\in\binom{[n]}{m}}\Delta_I,\qquad D_{m,n}=V(f_{m,n})\subset\C^{m\times n},\qquad U_{m,n}=\C^{m\times n}\setminus D_{m,n}.\]
Thus $U_{m,n}$ is the space of $n$-tuples of vectors in $\C^m$ in general position (every $m$ of them are a basis). The following is known about the geometry of $D_{m,n}$: 
\begin{itemize}[leftmargin=2em]
\item $D_{1,n}$ is the normal crossing divisor $x_1\cdots x_n=0$ and hence this is a classical Boolean arrangement.
\item $D_{m,m+1}$ is a \emph{linear} free divisor by \cite{BM}.
\item $D_{2,n}$ is free for all $n\ge3$ by \cite[Thm.~3.3]{Yim}. A basis consists of $n+3$ linear vector fields and $n-3$ vector fields of degree $n+1$. Yim also computed $\Poin(U_{2,n},t)$ \cite[Thm.~4.1]{Yim}, namely he proved that $\pi_i(U_{2,n})\cong\pi_i(S^3)$ for $i\ge2$ \cite[Thm.~4.4]{Yim}, and studied subarrangements indexed by graphs.
\end{itemize}
Yim asked \cite[Question~3.4]{Yim} whether $D_{m,n}$ is free for all $n>m>2$, and remarked that already the case $3\times5$ was open. He also observed \cite[Example~4.6]{Yim} that his fibration method for computing the topology of the complement breaks down for $3\times7$ matrices.

The $3\times5$ setting was also considered in \cite{JanaszWisniewska}, where determinantal arrangements built from selected products of the $3\times3$ minors were studied and several natural examples were shown to be free. The present paper concerns instead the full product of all ten maximal minors and shows that the resulting divisor is not free.

Our first observation is elementary but organizes everything our work in the paper. Since on $U_{m,n}$ the first $m$ columns form a basis, $U_{m,n}$ is a \emph{trivial} principal $\GL_m$-bundle over the open Grassmann stratum, and the base is cut out, in the affine coordinates of a big cell, by the nonvanishing of \emph{all} minors of an $m\times(n-m)$ matrix. Because a matrix and its transpose have the same minors, this makes Gale duality (association) visible at the level of the arrangement, both globally and locally.

Here is our main result of the paper. We will denote by $\mathfrak{P}_{i}$ the pure braid group on $i$ strands.
\begin{mainthm}
\label{thmA}
For all $1\le m<n$ there is an isomorphism of varieties $U_{m,n}\cong\GL_m(\C)\times Y_{m,n}$ with $Y_{m,n}\cong\Gr(m,n)^\circ\cong\Gr(n-m,n)^\circ\cong Y_{n-m,n}$. If $m\in\{2,n-2\}$, then
\[U_{m,n}\;\cong\;\GL_m(\C)\times(\C^*)^{n-1}\times\Mo{}_{,n}.\]
In particular,
\begin{gather*}
U_{3,5}\cong\GL_3(\C)\times(\C^*)^4\times\Mo{}_{,5}\simeq\SU(3)\times T^5\times\Mo{}_{,5},\\
\Poin(U_{3,5},t)=(1+t)^5(1+t^3)(1+t^5)(1+2t)(1+3t)=(1+t^5) \cdot \Poin(U_{2,5},t),\\
\pi_i(U_{3,5})\cong\pi_i(\SU(3)) \text{ for } i\ge2 \text{ and } \pi_1(U_{3,5})\cong\Z^4\times \mathfrak{P}_4.
\end{gather*}
\end{mainthm}

For $m=2$, our Theorem~\ref{thmA} recovers \cite[Theorems~4.1 and~4.4]{Yim} for the complete graph and adds $\pi_1(U_{2,n})\cong\Z^{n-1}\times \mathfrak{P}_{n-1}$ (Remark~\ref{rem:Yimtop}). Our second result answers Yim's question for $3\times5$ in the negative.

\begin{mainthm}
\label{thmB}
The divisor $D_{3,5}\subset\C^{15}$ is not free. More precisely, let $R=\C[x_{rk}]$, then
\begin{enumerate}[label=\textup{(\arabic*)},leftmargin=2em]
\item $\Der(-\log D_{3,5})$ is minimally generated by the $13$ linear vector fields coming from the action of $\GL_3\times(\C^*)^5$ and by five vector fields $\Phi_1,\dots,\Phi_5$ of degree $10$, given explicitly in \eqref{eq:psi} and Definition~\ref{def:Phi};
\item its minimal graded free resolution is
\[0\longrightarrow R(-11)^3\longrightarrow R(-1)^{13}\oplus R(-10)^5\longrightarrow\Der(-\log D_{3,5})\longrightarrow 0,\]
and the three syzygies are of the form $\sum_{k=1}^5x_{rk}\Phi_k\in R\cdot\cL$, $r=1,2,3$, where $\cL$ is the space of linear logarithmic fields. Hence $\Der(-\log D_{3,5})/R\cL\cong\coker\big(M^T\colon R(-11)^3\to R(-10)^5\big)$;
\item the stalk of $\Der(-\log D_{3,5})$ at a point $p$ needs exactly $18-\rk M(p)$ generators. In particular the non-free locus of $D_{3,5}$ is the determinantal variety $\{\rk M\le 2\}$, of codimension $3$.
\end{enumerate}
\end{mainthm}

The proof of the non-freeness itself only needs a Gr\"obner basis computation in $9$ variables (Proposition~\ref{prop:certificate}), combined with a slice at rank-two matrices (Proposition~\ref{prop:chart}). The remaining parts of Theorem~\ref{thmB} use a computation in $15$ variables. The computations for $3\times5$ case were preformed with \emph{Macaulay2} \cite{M2}, and the complete script, including all other sizes, runs in about some minutes -- the code is described in Appendix~\ref{app:code}.

The geometric picture behind Theorem~\ref{thmB} is described by the next statement.

\begin{mainthm}\label{thmC}
\begin{enumerate}[label=\textup{(\arabic*)},leftmargin=2em]
\item At a point of rank $m$, the germ of $D_{m,n}$ is isomorphic, up to a smooth factor, to the germ of $D_{n-m,n}$ at a point of rank $n-m$. Consequently $D_{n-2,n}$ is free at all points of rank $n-2$; in particular $D_{3,5}$ is free at all matrices of rank $3$.
\item Near a rank-two matrix whose columns are pairwise independent, $D_{3,5}$ is (up to a smooth factor) a two-parameter family of central arrangements $\cA(P_1,\dots,P_5)\subset\C^3$ of the ten planes $\operatorname{span}(P_i,P_j)$, where $P_1,\dots,P_5\in\C^3$ are in general position and are the Gale transform of the five columns; the minor $\Delta_I$ corresponds to the plane spanned by $P_j,P_k$ with $\{j,k\}=[5]\setminus I$.
\item Each arrangement $\cA(P_1,\dots,P_5)$ has characteristic polynomial $(t-1)(t^2-9t+21)$ and is therefore not free.
\end{enumerate}
\end{mainthm}

\subsection*{Remarks}
\begin{enumerate}[label=(\alph*),leftmargin=2em]
\item \emph{Freeness is invisible in the topology.} $D_{2,5}$ is free and $D_{3,5}$ is not, yet by Theorem~\ref{thmA} their complements differ only by the factor $\GL_3$ versus $\GL_2$, and both Poincar\'e polynomials factor as a product of the Poincar\'e polynomial of a general linear group and linear factors. So the analogy with Terao's theorem that motivated \cite{Yim} breaks down already for $3\times5$ matrices.
\item \emph{The obstruction reappears one step later.} The quadratic $t^2-9t+21$ that obstructs freeness of the transversal arrangements of $D_{3,5}$ also appears in the topology of $U_{3,6}$: the fibre of the projection $U_{3,6}\to U_{3,5}$ forgetting the last column is the complement of $\cA(v_1,\dots,v_5)$, and
\[\Poin(U_{3,6},t)=(1+t)^6(1+t^3)(1+t^5)(1+2t)(1+3t)(1+9t+21t^2)\]
by Proposition~\ref{prop:36}, which has a factor irreducible over $\Q$.
\item \emph{Other sizes.} The chart of Theorem~\ref{thmC}(2) exists for every $m\times n$. Computing logarithmic vector fields in this chart shows that $D_{3,6}$, $D_{4,6}$ and $D_{5,7}$ are not free either (Section~\ref{sec:other}). This supports Conjecture~\ref{conj}: \emph{$D_{m,n}$ is free if and only if $m\le2$ or $n=m+1$.}
\end{enumerate}

\subsection*{Organization}
Section~\ref{sec:prelim} collects notation and standard facts; it also introduces the logarithmic vector fields $\psi_{i\to k}$ that exist for every $m\times n$. Section~\ref{sec:slice} contains the slice construction and Gale duality. Section~\ref{sec:topology} treats the topology of the complements, including the ten-plane arrangement and $U_{3,6}$. Section~\ref{sec:local} describes the local structure of $D_{m,n}$, and Section~\ref{sec:35} proves Theorem~\ref{thmB}. Section~\ref{sec:other} reports on other sizes. Appendix~\ref{app:code} contains the Macaulay2 code. Appendix~\ref{app:2n} gives a short proof of the freeness of $D_{2,n}$, which is used in Theorem~\ref{thmC}(1).

\section{Preliminaries}\label{sec:prelim}

\subsection{Notation}
Throughout, $1\le m<n$, $R=\C[x_{rk}\mid 1\le r\le m,\ 1\le k\le n]$, $M=(x_{rk})$ is the generic matrix with columns $v_1,\dots,v_n$, and $\Delta_I=\det M_I$ for $I\in\binom{[n]}{m}$, the columns of $M_I$ being taken in increasing order. For vectors $w_1,\dots,w_m\in\C^m$ we write $[w_1,\dots,w_m]=\det(w_1|\cdots|w_m)$. For $v\in\C^m$ and a column index $k$ we write
\[v\cdot\dd{v_k}:=\sum_{r=1}^m v^{(r)}\frac{\partial}{\partial x_{rk}},\]
so that $(v\cdot\dd{v_k})(\Delta_I)=0$ if $k\notin I$, and otherwise equals the determinant obtained from $M_I$ by replacing the column $v_k$ with $v$.

A polynomial vector field $\theta=\sum g_{rk}\,\partial/\partial x_{rk}$ is \emph{homogeneous of degree $d$} if all $g_{rk}$ are homogeneous of degree $d$.
For a subset $P\subset[m]$ and $Q\subset[n-m]$ with $|P|=|Q|$ and a matrix $W\in\C^{m\times(n-m)}$, $W_{P,Q}$ denotes the submatrix with rows $P$ and columns $Q$. By convention $\det W_{\emptyset,\emptyset}=1$.

\subsection{Logarithmic vector fields}
For a reduced polynomial $f\in S=\C[y_1,\dots,y_N]$ put $\Der(-\log f)=\{\theta\in\Der_S\mid\theta(f)\in fS\}$. If $f=\prod_i f_i$ with pairwise non-associate irreducible $f_i$, then $\Der(-\log f)=\bigcap_i\{\theta\mid\theta(f_i)\in f_iS\}$. The divisor $V(f)$ is free at $p$ if the localization (or, equivalently, the analytic stalk) $\Der(-\log f)_p$ is a free $\mathcal O_p$-module of rank $N$. We use the following standard facts that are well-known among the experts. However, we recall them for the completness.

\begin{lemma}\label{lem:standard}
Let $f\in S$ be reduced.
\begin{enumerate}[label=\textup{(\alph*)},leftmargin=2em]
\item \textup{(Saito's criterion \cite{Saito})} Elements $\theta_1,\dots,\theta_N\in\Der(-\log f)$ form a basis if and only if $\det(\theta_i(y_j))=c\cdot f$ with $c$ a unit. For arbitrary $\theta_i\in\Der(-\log f)$ one always has $\det(\theta_i(y_j))\in fS$.
\item If $f$ is homogeneous and $\theta_1,\dots,\theta_N\in\Der(-\log f)$ are homogeneous with $\sum_i\deg\theta_i=\deg f$ \textup{(}in the convention above\textup{)} and $\det(\theta_i(y_j))\neq0$, then they form a basis. If $V(f)$ is free, the degrees of any homogeneous basis add up to $\deg f$.
\item The set of points where $V(f)$ is not free is Zariski closed. If $f$ is homogeneous, $V(f)$ is free everywhere if and only if it is free at $0$, if and only if the graded module $\Der(-\log f)$ is free.
\item If $\alpha\colon(\C^N,p)\to(\C^{N-k},p')\times(\C^k,0)$ is a local analytic isomorphism taking $V(f)$ to $V(f')\times\C^k$, then $\Der(-\log f)_p\cong\Der(-\log f')_{p'}\oplus\mathcal O^k$. In particular, $V(f)$ is free at $p$ if and only if $V(f')$ is free at $p'$.
\item If $\Der(-\log f)$ is generated by $g$ elements with presentation matrix $P$, then the stalk at $p$ is minimally generated by $g-\rk P(p)$ elements. Here $P(p)$ is the scalar matrix obtained by evaluating $P$ at $p$. In particular, $V(f)$ is not free at $p$ if and only if $\rk P(p)<g-N$.
\item Formation of $\Der(-\log f)$ commutes with flat base change. In particular, if $f$ has rational coefficients, computations over $\Q$ give the answer over $\C$.
\end{enumerate}
\end{lemma}
\begin{proof}[Comments on the proof]
Item (a) is \cite[(1.8)]{Saito}. The second assertion holds because at a smooth point of $V(f)$ all $\theta_i$ are tangent to $V(f)$, so the determinant vanishes on $V(f)$, and $f$ is reduced. Item (b) follows from (a), since $\det(\theta_i(y_j))$ is then a nonzero multiple of $f$ of the same degree. Item (d) holds because $\Der(-\log(D\times\C^k))=\Der(-\log D)\boxtimes\mathcal O\oplus\bigoplus_j\mathcal O\,\partial_{z_j}$. Finally, observe that item (e) is Nakayama's lemma, and item (f) holds because $\Der(-\log f)$ is the kernel of a map of free modules.
\end{proof}

\subsection{Symmetries and linear logarithmic fields}
The group $G=\GL_m(\C)\times(\C^*)^n$ acts on $\C^{m\times n}$ by $(g,t)\cdot M=gM\operatorname{diag}(t)$, and $\Delta_I((g,t)\cdot M)=\det(g)\prod_{k\in I}t_k\,\Delta_I(M)$. Hence $G$ preserves every component of $D_{m,n}$. The symmetric group $S_n$, permuting columns, preserves $D_{m,n}$ and permutes its components. Infinitesimally we obtain the linear logarithmic vector fields
\[\xi_A=\sum_{r,s=1}^m\sum_{k=1}^nA_{rs}\,x_{sk}\frac{\partial}{\partial x_{rk}}\ \ (A\in\mathfrak{gl}_m),\qquad \eta_k=\sum_{r=1}^mx_{rk}\frac{\partial}{\partial x_{rk}}\ \ (1\le k\le n),\]
with $\xi_A(\Delta_I)=\operatorname{tr}(A)\Delta_I$ and $\eta_k(\Delta_I)=[k\in I]\,\Delta_I$. Let $\cL_{m,n}$ denote their $\C$-span. The only linear relation is $\xi_{\mathrm{id}}=\sum_k\eta_k$ (both equal the Euler field).

\begin{lemma}\label{lem:linear}
If $n\ge m+1$ then $\dim\cL_{m,n}=m^2+n-1$, and any basis of $\cL_{m,n}$ is linearly independent over $R$.
\end{lemma}
\begin{proof}
It suffices to show that the fields are independent at a point $p\in U_{m,n}$, i.e.\ that the orbit $G\cdot p$ has dimension $m^2+n-1$. If $(g,t)\cdot p=p$ then $g$ preserves the $n\ge m+1$ lines $\C v_k$, which are in general position, so $g$ is a scalar $\lambda$; then $t_k=\lambda^{-1}$ for all $k$. The stabilizer is therefore one-dimensional.
\end{proof}

\subsection{The vector fields $\psi_{i\to k}$}
For distinct column indices $i\ne k$ put
\begin{equation}\label{eq:psi}
c_{i\to k}=\prod_{\substack{I\ni i\\ I\not\ni k}}\Delta_I,\qquad \psi_{i\to k}=c_{i\to k}\,(v_k\cdot\dd{v_i}).
\end{equation}
Thus $\psi_{i\to k}$ moves the column $v_i$ in the direction of $v_k$, damped by all minors that this motion would destroy. Equivalently, $\psi_{i\to k}(M)=c_{i\to k}(M)\cdot ME_{ki}$, where $E_{ki}$ is the elementary $n\times n$ matrix.

Grade $R$ by $\Z^m\times\Z^n$ via $\deg x_{rk}=(e_r,e_k)$, and give $\partial/\partial x_{rk}$ the degree $-(e_r,e_k)$. Then every $\Delta_I$ is homogeneous, $\Der(-\log f_{m,n})$ is multigraded, and the fields in $\cL_{m,n}$ have degree $(e_s-e_r,0)$ ($\xi_{E_{rs}}$) or $0$ ($\eta_k$).

\begin{lemma}\label{lem:psi}
$\psi_{i\to k}\in\Der(-\log f_{m,n})$. It is homogeneous of degree $m\binom{n-2}{m-1}+1$ and $\SL_m$-equivariant. Its multidegree is
\[\Big(\tbinom{n-2}{m-1}(1,\dots,1)\ ,\ \big(\tbinom{n-2}{m-1}-1\big)e_i+\tbinom{n-3}{m-2}\textstyle\sum_{j\ne i,k}e_j+e_k\Big).\]
This multidegree is independent of $i$ if and only if $\binom{n-3}{m-1}=1$, i.e.\ if and only if $n=m+2$ (or $m=1$).
\end{lemma}
\begin{proof}
If $i\notin I$ then $\psi_{i\to k}(\Delta_I)=0$. If $i,k\in I$, replacing $v_i$ by $v_k$ in $M_I$ produces two equal columns, so again $\psi_{i\to k}(\Delta_I)=0$. If $i\in I\not\ni k$, then $\Delta_I$ divides $c_{i\to k}$ and hence divides $\psi_{i\to k}(\Delta_I)$. The degree is $m$ times the number $\binom{n-2}{m-1}$ of factors of $c_{i\to k}$, plus one. Equivariance holds because $c_{i\to k}$ is $\SL_m$-invariant and $M\mapsto ME_{ki}$ is equivariant. For the multidegree, the column $i$ occurs in all $\binom{n-2}{m-1}$ factors of $c_{i\to k}$, a column $j\ne i,k$ occurs in $\binom{n-3}{m-2}$ of them, and $v_k\cdot\dd{v_i}$ contributes $e_k-e_i$. The last assertion follows from Pascal's rule $\binom{n-2}{m-1}=\binom{n-3}{m-2}+\binom{n-3}{m-1}$.
\end{proof}

\begin{example}\label{ex:psi35}
For $m=3$, $n=5$ each $c_{i\to k}$ is a product of three minors, so $\deg\psi_{i\to k}=10$. For example $\psi_{1\to5}=\Delta_{123}\Delta_{124}\Delta_{134}\,(v_5\cdot\dd{v_1})$. In terms of the multigrading by $\Z^3\times\Z^5$ (row degrees, column degrees), $\psi_{i\to k}$ has degree $\big((3,3,3),(2,2,2,2,2)-e_k\big)$, independently of $i$. For $m=2$, $n=4$, the field $\psi_{4\to1}=\Delta_{24}\Delta_{34}(v_1\cdot\dd{v_4})$ is the degree-$5$ basis element of \cite[Ex.~1.2]{Yim}.
\end{example}

\subsection{Hyperplane arrangements}
For a essential arrangement $\cA$ of hyperplanes in $\C^\ell$ with intersection lattice $L(\cA)$ and M\"obius function $\mu$, the characteristic polynomial is $\chi_\cA(t)=\sum_{X\in L(\cA)}\mu(X)t^{\dim X}$. The complement $M(\cA)$ has Poincar\'e polynomial $(-t)^\ell\chi_\cA(-1/t)$, and its cohomology ring is the Orlik--Solomon algebra, generated in degree one by the classes $d\log\alpha_H$ \cite{OT}. Over a finite field one has $|M(\cA)(\F_q)|=\chi_\cA(q)$ whenever the lattice is unchanged by reduction.

\begin{theorem}[Terao's factorization theorem \cite{Terao81}]\label{thm:terao}
If $\cA$ is free with exponents $d_1,\dots,d_\ell$, i.e., the degrees of a homogeneous basis of $\Der(-\log\cA)$, then $\chi_\cA(t)=\prod_{i=1}^\ell(t-d_i)$.
\end{theorem}

\subsection{Families of arrangements with constant lattice}
We need the following form, in families, of the principle that the topology of an arrangement complement depends only on the intersection lattice. For one-parameter families this is Randell's theorem \cite{Randell}. We deduce the version over an arbitrary base from Thom's first isotopy lemma, following \cite[Ch.~1]{Dimca}.

\begin{lemma}\label{lem:isotopy}
Let $B$ be a smooth manifold, $\ell\ge1$, and let $\alpha_1,\dots,\alpha_N\colon B\to(\C^\ell)^*\setminus\{0\}$ be smooth maps \textup{(}for instance, $B$ a complex manifold and the $\alpha_i$ holomorphic\textup{)}. For $b\in B$ let $\cA(b)$ be the central arrangement of the hyperplanes $H_i(b)=\ker\alpha_i(b)\subset\C^\ell$. Assume that the rank function
\[r(S)=\rk\{\alpha_i(b)\mid i\in S\}=\operatorname{codim}\textstyle\bigcap_{i\in S}H_i(b),\qquad S\subset[N],\]
does not depend on $b$. Put
\[\mathcal M=\{(b,x)\in B\times\C^\ell\mid \alpha_i(b)(x)\ne0\ \text{for } 1\le i\le N\}.\]
Then $\mathrm{pr}_1\colon\mathcal M\to B$ is a locally trivial fibration whose fibre over $b$ is $M(\cA(b))$. That is, every $b_0\in B$ has an open neighbourhood $V$ with a homeomorphism $\varphi\colon M(\cA(b_0))\times V\to\mathcal M\cap(V\times\C^\ell)$ such that $\mathrm{pr}_1\circ\varphi=\mathrm{pr}_2$.
\end{lemma}
\begin{proof}
We use stratifications, Whitney's condition (b) and Whitney stratifications as defined in \cite[Ch.~1, (1.5)--(1.7)]{Dimca}. Complex manifolds are regarded as real ones.

\emph{Step 1: compactifying the fibres.} The projection $\mathcal M\to B$ is not proper. So, as in the proof of \cite[Ch.~1, (4.1)]{Dimca}, we add the hyperplane at infinity. Let $L=\ell+1$ and $P=\PP(\C\oplus\C^\ell)\cong\PP^\ell$, with homogeneous coordinates $[x_0:x]$. Consider the linear forms on $\C^L$
\[\beta_0=x_0,\qquad \beta_i(b)(x_0,x)=\alpha_i(b)(x)\quad(1\le i\le N).\]
The forms $\beta_1(b),\dots,\beta_N(b)$ do not involve $x_0$. Hence the rank function $\rho(S)=\rk\{\beta_i(b)\mid i\in S\}$, $S\subset\{0,\dots,N\}$, equals $r(S\setminus\{0\})+1$ if $0\in S$ and $r(S)$ otherwise, so it is also independent of $b$. Let $\mathcal E=B\times P$ and $f=\mathrm{pr}_1\colon\mathcal E\to B$. Since $P$ is compact, $f$ is proper. The map $(b,x)\mapsto(b,[1:x])$ is a diffeomorphism of $\mathcal M$ onto the open set
\[\mathcal M'=\{(b,[y])\in\mathcal E\mid \beta_i(b)(y)\ne0\ \text{for } 0\le i\le N\}\]
which is compatible with the projections to $B$, and it identifies $M(\cA(b))$ with $\mathcal M'\cap f^{-1}(b)$. It therefore suffices to prove the lemma for $f|_{\mathcal M'}$.

\emph{Step 2: the moving flats.} For $S\subset\{0,\dots,N\}$ let $\widehat X_S(b)=\bigcap_{i\in S}\ker\beta_i(b)\subset\C^L$, a subspace of dimension $L-\rho(S)$, and let
\[\mathcal X_S=\{(b,[y])\in\mathcal E\mid y\in\widehat X_S(b)\},\]
a closed subset of $\mathcal E$. We claim the following. If $s=\rho(S)<L$, then every $b_0\in B$ has an open neighbourhood $V$ and a smooth map $G\colon V\to\GL_L(\C)$ such that the diffeomorphism
\[\Theta\colon V\times P\to V\times P,\qquad (b,[w])\mapsto(b,[G(b)w]),\]
maps $V\times\Lambda$ onto $\mathcal X_S\cap(V\times P)$, where $\Lambda=\PP(\C^{L-s}\times0)$ is a fixed linear subspace of $P$. To see this, choose $i_1,\dots,i_s\in S$ such that $\beta_{i_1}(b_0),\dots,\beta_{i_s}(b_0)$ are linearly independent. They stay independent on a neighbourhood $V$ of $b_0$. Since $\rho(S)=s$ for every $b$, they span $\langle\beta_i(b)\mid i\in S\rangle$, so $\widehat X_S(b)=\ker A(b)$, where $A(b)$ is the $s\times L$ matrix with rows $\beta_{i_1}(b),\dots,\beta_{i_s}(b)$. After a fixed permutation of the coordinates of $\C^L$ and after shrinking $V$, we can write $A(b)=(A_1(b)\,|\,A_2(b))$ with $A_1(b)\in\GL_s(\C)$. Then
\[G(b)=\begin{pmatrix}-A_1(b)^{-1}A_2(b)& I_s\\ I_{L-s}&0\end{pmatrix}\]
satisfies $\det G(b)=\pm1$, and its first $L-s$ columns form a basis of $\ker A(b)$. This proves the claim. Consequently $\mathcal X_S$ is a closed submanifold of $\mathcal E$, and $f|_{\mathcal X_S}$ is a submersion, because in the chart $\Theta$ it becomes the projection $V\times\Lambda\to V$. If $\rho(S)=L$, then $\mathcal X_S=\emptyset$.

\emph{Step 3: the stratification.} For $(b,[y])\in\mathcal E$ let $F(b,[y])=\{i\mid \beta_i(b)(y)=0\}$. Call $F\subset\{0,\dots,N\}$ \emph{closed} if $\rho(F\cup\{i\})>\rho(F)$ for all $i\notin F$. Let $\mathcal F$ be the set of closed $F$ with $\rho(F)<L$. It is finite and independent of $b$. Note that $\rho(F\cup\{i\})=\rho(F)$ if and only if $\widehat X_F(b)\subset\ker\beta_i(b)$. Since $0\ne y\in\widehat X_{F(b,[y])}(b)$, it follows that every $F(b,[y])$ lies in $\mathcal F$. For $F\in\mathcal F$ put
\[\mathcal S_F=\{(b,[y])\in\mathcal E\mid F(b,[y])=F\}=\mathcal X_F\setminus\bigcup_{i\notin F}\mathcal X_{F\cup\{i\}}.\]
The finitely many sets $\mathcal S_F$ partition $\mathcal E$. Each $\mathcal S_F$ is open in $\mathcal X_F$, so by Step~2 it is a submanifold of $\mathcal E$ on which $f$ is a submersion. Discarding empty sets, $\mathcal S=\{\mathcal S_F\}_{F\in\mathcal F}$ is therefore a stratification of $\mathcal E$ in the sense of \cite[Ch.~1, (1.5)]{Dimca}. Since all $\beta_i\ne0$, we have $\emptyset\in\mathcal F$, and the open stratum is $\mathcal S_\emptyset=\mathcal M'$. On each fibre $f^{-1}(b)=P$, $\mathcal S$ induces the stratification of $P$ by the flats of the projective arrangement $\{\PP(\ker\beta_i(b))\}_{0\le i\le N}$.

\emph{Step 4: Whitney regularity.} Let $F\ne G$ in $\mathcal F$ and $p=(b_0,[y_0])\in\mathcal S_G\cap\overline{\mathcal S_F}$. As $\mathcal X_F$ is closed, $p\in\mathcal X_F$, i.e.\ $F\subset F(p)=G$, and therefore $\mathcal S_G\subset\mathcal X_G\subset\mathcal X_F$. Choose $V$ and $\Theta$ as in Step~2 for $S=F$ and $b_0$. Condition (1.6) of \cite[Ch.~1]{Dimca} is invariant under diffeomorphisms (see the remark following \cite[Ch.~1, (1.6)]{Dimca}). It therefore suffices to check it for $Y=\Theta^{-1}(\mathcal S_F)$ over $X=\Theta^{-1}(\mathcal S_G)$ at $q=\Theta^{-1}(p)$. Here $Y$ is open in $V\times\Lambda$ and $X\subset V\times\Lambda$. Take a chart of $V$ at $b_0$ and an affine chart of $P$ centred at the second component of $q$; in the latter, $\Lambda$ becomes a linear subspace. In these coordinates we work in an open subset of $\mathbb R^m$ with $q=0$, and $V\times\Lambda$ corresponds to an open subset of a linear subspace $E\subset\mathbb R^m$, with $Y$ open in $E$ and $X\subset E$. Let $x_n\in X$ and $y_n\in Y$ be sequences converging to $0$ as in \cite[Ch.~1, (1.6)]{Dimca}. Then $T_{y_n}Y=E$ for all $n$, so $T=E$. Every secant direction $y_n-x_n$ lies in $E$, so the limit line $l$ lies in $E=T$. This is condition (1.6). (It is the argument behind \cite[Ch.~1, (1.8)(i)]{Dimca}.) Hence $\mathcal S$ is a Whitney stratification of $\mathcal E$ in the sense of \cite[Ch.~1, (1.7)]{Dimca}.

\emph{Step 5: Thom's first isotopy lemma.} The map $f$ is smooth and proper, so $f|_{\overline{\mathcal S_F}}$ is proper for every $F$. By Step~3, $f|_{\mathcal S_F}$ is a submersion. We apply Thom's first isotopy lemma \cite[Ch.~1, Thm.~(3.5)]{Dimca} with $M=A=\mathcal E$ and the Whitney stratification $\mathcal S$. It shows that every $b_0\in B$ has an open neighbourhood $V$ and a homeomorphism $\varphi\colon f^{-1}(b_0)\times V\to f^{-1}(V)$ with $f\circ\varphi=\mathrm{pr}_2$ that preserves the strata, i.e.\ maps $(\mathcal S_F\cap f^{-1}(b_0))\times V$ onto $\mathcal S_F\cap f^{-1}(V)$ for every $F\in\mathcal F$. For $F=\emptyset$ this is a trivialization of $f|_{\mathcal M'}$ over $V$ with fibre $\mathcal M'\cap f^{-1}(b_0)\cong M(\cA(b_0))$. By Step~1, this proves the lemma.

The constancy of $r$ was used only in Step~2. Once the moving flats are submanifolds that submerse onto $B$, Whitney regularity is automatic: each stratum is open in a submanifold that contains every stratum meeting its closure.
\end{proof}

\subsection{The moduli space $\Mo{}_{,n}$}
For $n\ge3$ let $\Mo{}_{,n}$ be the moduli space of $n$ distinct ordered points on $\PP^1$. Fixing the first three points at $\infty,0,1$ identifies $\Mo{}_{,n}$ with $\Conf_{n-3}(\C\setminus\{0,1\})$. The Fadell--Neuwirth fibrations \cite{FN} show that $\Mo{}_{,n}$ is a $K(\pi,1)$ with
\[\Poin(\Mo{}_{,n},t)=\prod_{k=2}^{n-2}(1+kt),\qquad |\Mo{}_{,n}(\F_q)|=\prod_{k=2}^{n-2}(q-k).\]
Moreover $\Conf_{n-1}(\C)\cong\C\times\C^*\times\Mo{}_{,n}$, since the affine group acts freely and transitively on the fibres of $\Conf_{n-1}(\C)\to\Mo{}_{,n}$ and a section is given by normalizing two points to $0,1$. Hence $\mathfrak{P}_{n-1}\cong\Z\times\pi_1(\Mo{}_{,n})$, where $\mathfrak{P}_{n-1}$ is the pure braid group \cite{Arnold}.

\section{The \texorpdfstring{$\GL_m$}{GLm}-slice and Gale duality}\label{sec:slice}

\begin{lemma}\label{lem:minorsIW}
Let $W\in\C^{m\times(n-m)}$ and $N=[I_m\,|\,W]$. For $I=I_1\sqcup(m+I_2)$ with $I_1\subset[m]$ and $I_2\subset[n-m]$ one has
\[\Delta_I(N)=\pm\det W_{[m]\setminus I_1,\,I_2}.\]
The assignment $I\mapsto([m]\setminus I_1,I_2)$ is a bijection between $\binom{[n]}{m}$ and the set of pairs $(P,Q)$, $P\subset[m]$, $Q\subset[n-m]$, $|P|=|Q|$. Thus the maximal minors of $[I_m|W]$ are, up to sign, exactly the minors of $W$ of all sizes, each occurring once, the minor of size $0$ being $\Delta_{[m]}(N)=1$.
\end{lemma}
\begin{proof}
Expand $\det N_I$ along the columns $e_i$, $i\in I_1$: each such column has a single entry $1$, in row $i$. The number of pairs $(P,Q)$ equals $\sum_k\binom mk\binom{n-m}k=\binom nm$.
\end{proof}

For $W\in\C^{m\times(n-m)}$ let
\begin{gather*}
\mu_{m,n}(W)=\prod_{\substack{P\subset[m],\,Q\subset[n-m]\\ 1\le|P|=|Q|}}\det W_{P,Q},\\
E_{m,n}=V(\mu_{m,n})\subset\C^{m\times(n-m)},\qquad Y_{m,n}=\C^{m\times(n-m)}\setminus E_{m,n}.
\end{gather*}
So $Y_{m,n}$ is the variety of $m\times(n-m)$ matrices all of whose minors are nonzero.

\begin{proposition}\label{prop:slice}
Let $\Omega=\{\Delta_{[m]}\ne0\}\subset\C^{m\times n}$. The map
\[\Phi\colon\GL_m\times\C^{m\times(n-m)}\longrightarrow\Omega,\qquad(g,W)\longmapsto g\,[I_m\,|\,W],\]
is an isomorphism of varieties, with inverse
\[M\longmapsto\big(M_{[m]},\,M_{[m]}^{-1}M_{[n]\setminus[m]}\big),\]
and $\Phi^*f_{m,n}=\pm\det(g)^{\binom nm}\mu_{m,n}(W)$. Consequently, $\Phi^{-1}(D_{m,n}\cap\Omega)=\GL_m\times E_{m,n}$ and
\[U_{m,n}\cong\GL_m\times Y_{m,n}.\]
The same holds over $\Z$, hence over every field. For any $I\in\binom{[n]}m$ the analogous statement holds on $\{\Delta_I\ne0\}$, after permuting columns.
\end{proposition}
\begin{proof}
Since $\Delta_I(gN)=\det(g)\Delta_I(N)$, this follows from Lemma~\ref{lem:minorsIW}. Moreover $U_{m,n}\subset\Omega$.
\end{proof}

\begin{proposition}[Gale duality]\label{prop:gale}
\begin{enumerate}[label=\textup{(\alph*)},leftmargin=2em]
\item Transposition $W\mapsto W^T$ maps $E_{m,n}$ isomorphically onto $E_{n-m,n}$ and $Y_{m,n}$ onto $Y_{n-m,n}$.
\item The map $W\mapsto\operatorname{rowspace}[I_m|W]$ is an isomorphism of $Y_{m,n}$ onto the open subset $\Gr(m,n)^\circ\subset\Gr(m,n)$ where all Pl\"ucker coordinates are nonzero.
\item Under \textup{(b)}, the isomorphism $W\mapsto-W^T$ \textup{(}the composite of \textup{(a)} with $W\mapsto -W$\textup{)} corresponds to $\Lambda\mapsto\Lambda^\perp$, followed by the cyclic renumbering of the coordinates that puts the last $n-m$ first. This is the classical association, or Gale transform, of point configurations \cite{DO,GGMS}.
\end{enumerate}
\end{proposition}
\begin{proof}
Item (a) follows from the fact that the minors of $W^T$ are the minors of $W$. Then for (b) we observe that the Pl\"ucker coordinates of $\operatorname{rowspace}[I_m|W]$ are the maximal minors of $[I_m|W]$, i.e.\ the minors of $W$ by Lemma~\ref{lem:minorsIW}; a subspace $\Lambda$ with nonzero coordinate $p_{[m]}$ has a unique representative of the form $[I_m|W]$. Finally, item (c) follows from the fact that $\operatorname{rowspace}[I_m|W]^\perp=\operatorname{rowspace}[-W^T|I_{n-m}]$.
\end{proof}

\begin{proposition}\label{prop:Y2n}
For $n\ge3$ there are isomorphisms $Y_{n-2,n}\cong Y_{2,n}\cong(\C^*)^{n-1}\times\Mo{}_{,n}$, defined over $\Z$.
\end{proposition}
\begin{proof}
Let $W\in Y_{2,n}$ have columns $(a_j,b_j)^T$, $1\le j\le n-2$. The conditions defining $Y_{2,n}$ are $a_j,b_j\ne0$ and $a_ib_j-a_jb_i=a_ia_j(s_j-s_i)\ne0$, where $s_j=b_j/a_j$. So $W\mapsto\big((a_j)_j,(s_j)_j\big)$ identifies $Y_{2,n}$ with $(\C^*)^{n-2}\times\Conf_{n-2}(\C^*)$. Next, $(s_1,\dots,s_{n-2})\mapsto(s_1;s_2/s_1,\dots,s_{n-2}/s_1)$ identifies $\Conf_{n-2}(\C^*)$ with $\C^*\times\Conf_{n-3}(\C\setminus\{0,1\})=\C^*\times\Mo{}_{,n}$. Finally, $Y_{n-2,n}\cong Y_{2,n}$ by Proposition~\ref{prop:gale}(a).
\end{proof}

\begin{example}[$3\times5$ case]\label{ex:Y35}
Write $W=[w_4|w_5]\in\C^{3\times2}$ with $w_4=(a_1,a_2,a_3)^T$ and $w_5=(a_1r_1,a_2r_2,a_3r_3)^T$. Then $W\in Y_{3,5}$ if and only if $a_i\ne0$, $r_i\ne0$ and $r_i\ne r_j$ for $i\ne j$, since the $2\times2$ minors are $a_ia_j(r_j-r_i)$. Hence $Y_{3,5}\cong(\C^*)^3\times\Conf_3(\C^*)\cong(\C^*)^4\times\Mo{}_{,5}$. Geometrically, $\Gr(3,5)^\circ/(\C^*)^5$ is the moduli space of five points in general position in $\PP^2$, and Gale duality identifies it with $\Mo{}_{,5}$, the moduli of five distinct points on $\PP^1$ (equivalently, on the unique conic through the five points).
\end{example}

\begin{example}[$3\times 6$ case]\label{ex:Y36}
The torus $(\C^*)^3\times(\C^*)^3$ acts on $\C^{3\times3}$ by row and column scaling; the subtorus of pairs $(\lambda\cdot\mathbf 1,\lambda^{-1}\cdot\mathbf 1)$ acts trivially, and the quotient torus $(\C^*)^5$ acts freely on $Y_{3,6}$. Normalizing the first row and column of $W$ to $1$ gives $Y_{3,6}\cong(\C^*)^5\times X(3,6)$, where $X(3,6)$ is the moduli space of six ordered points in general position in $\PP^2$ (the columns of $[I_3|W]$).
\end{example}

\section{Topology of the complements}\label{sec:topology}

\begin{theorem}\label{thm:complement}
For $1\le m<n$ one has:
\begin{enumerate}[label=\textup{(\alph*)},leftmargin=2em]
\item $U_{m,n}\cong\GL_m(\C)\times Y_{m,n}$ and $Y_{m,n}\cong\Gr(m,n)^\circ\cong\Gr(n-m,n)^\circ\cong Y_{n-m,n}$;
\item $\Poin(U_{m,n},t)=\prod_{k=1}^m(1+t^{2k-1})\cdot\Poin(Y_{m,n},t)$. In particular,
\[\frac{\Poin(U_{m,n},t)}{\Poin(\GL_m,t)}=\frac{\Poin(U_{n-m,n},t)}{\Poin(\GL_{n-m},t)};\]
\item $U_{m,n}\simeq\SU(m)\times S^1\times Y_{m,n}$. Hence $\pi_i(U_{m,n})\cong\pi_i(\SU(m))\oplus\pi_i(Y_{m,n})$ for $i\ge2$ and $\pi_1(U_{m,n})\cong\Z\times\pi_1(Y_{m,n})$.
\end{enumerate}
\end{theorem}
\begin{proof}
Item (a) follows from Propositions~\ref{prop:slice} and~\ref{prop:gale}. Item (b) follows from the K\"unneth formula and $H^*(\GL_m(\C);\Z)=H^*(\U(m);\Z)=\Lambda_\Z(e_1,e_3,\dots,e_{2m-1})$. Finally, item (c) holds because $\GL_m(\C)$ deformation retracts onto $\U(m)$, and $\U(m)\cong\SU(m)\times S^1$ as manifolds via
\[A\longmapsto\big(A\operatorname{diag}(\det A^{-1},1,\dots,1),\ \det A\big).\] 
\end{proof}

\begin{corollary}\label{cor:nminus2}
Let $n\ge4$ and $m\in\{2,n-2\}$. Then $U_{m,n}\cong\GL_m(\C)\times(\C^*)^{n-1}\times\Mo{}_{,n}$, and
\begin{enumerate}[label=\textup{(\alph*)},leftmargin=2em]
\item $\displaystyle\Poin(U_{m,n},t)=\prod_{k=1}^m(1+t^{2k-1})\cdot(1+t)^{n-1}\cdot\prod_{k=2}^{n-2}(1+kt)$;
\item $\pi_i(U_{m,n})\cong\pi_i(\SU(m))$ for $i\ge2$; the universal cover of $U_{m,n}$ is homotopy equivalent to $\SU(m)$;
\item $\pi_1(U_{m,n})\cong\Z^{n}\times\pi_1(\Mo{}_{,n})\cong\Z^{n-1}\times \mathfrak{P}_{n-1}$;
\item $|U_{m,n}(\F_q)|=|\GL_m(\F_q)|\,(q-1)^{n-1}\prod_{k=2}^{n-2}(q-k)$.
\end{enumerate}
\end{corollary}
\begin{proof}
Combine Theorem~\ref{thm:complement} with Proposition~\ref{prop:Y2n} and the facts on $\Mo{}_{,n}$ recalled in \S\ref{sec:prelim}. The space $(\C^*)^{n-1}\times\Mo{}_{,n}$ is aspherical.
\end{proof}

\begin{remark}\label{rem:Yimtop}
For $m=2$, Corollary~\ref{cor:nminus2}(a) is \cite[Thm.~4.1]{Yim}: $(1+t)(1+t^3)(1+t)^{n-1}\prod_{k=2}^{n-2}(1+kt)=(1+t^3)(1+t)^{n-1}\prod_{k=1}^{n-2}(1+kt)$. Item (b) is \cite[Thm.~4.4]{Yim} for the complete graph, since $\SU(2)=S^3$. Item (c) answers \cite[Rem.~4.5]{Yim} in this case. Beyond standard facts about $\Mo{}_{,n}$, the proof needs no fibrations or spectral sequences. The fibration obstruction of \cite[Ex.~4.6]{Yim} for $3\times7$ matrices plays no role for the slice, which exists for all $m\times n$; what is not known in general is the topology of $\Gr(m,n)^\circ$.
\end{remark}

\begin{corollary}[$3\times5$ case]\label{cor:35top}
$U_{3,5}\cong\GL_3(\C)\times(\C^*)^4\times\Mo{}_{,5}\simeq\SU(3)\times T^5\times\Mo{}_{,5}$, and
\begin{align*}
\Poin(U_{3,5},t)&=(1+t)^5(1+t^3)(1+t^5)(1+2t)(1+3t)=(1+t^5)\,\Poin(U_{2,5},t)\\
&=1+10t+41t^2+91t^3+125t^4+128t^5+135t^6+162t^7+177t^8+160t^9\\
&\quad+133t^{10}+125t^{11}+121t^{12}+86t^{13}+35t^{14}+6t^{15}.
\end{align*}
Furthermore $\pi_i(U_{3,5})\cong\pi_i(\SU(3))$ for $i\ge2$ \textup{(}so $\pi_2=\pi_4=0$, $\pi_3\cong\pi_5\cong\Z$, $\pi_6\cong\Z/6$, \dots\textup{)}, $\pi_1(U_{3,5})\cong\Z^4\times \mathfrak{P}_4$ and $\chi(U_{3,5})=0$.
\end{corollary}

\subsection{The ten-plane arrangement}
Let us recall the following well-known construction.
\begin{proposition}\label{prop:tenplanes}
Let $P_1,\dots,P_5$ be vectors in $K^3$, $K$ is a field, any three of which are linearly independent. Let $\cA=\cA(P_1,\dots,P_5)$ be the central arrangement of the ten planes $H_{ij}=\operatorname{span}(P_i,P_j)$, $1\le i<j\le5$. Then
\begin{enumerate}[label=\textup{(\alph*)},leftmargin=2em]
\item the planes $H_{ij}$ are pairwise distinct;
\item the rank-two flats of $\cA$ are the five lines $KP_i$, each contained in exactly the four planes $H_{ij}$ \textup{(}$j\ne i$\textup{)}, and the fifteen lines $H_{ij}\cap H_{kl}$ with $\{i,j\}\cap\{k,l\}=\emptyset$, each contained in exactly two planes;
\item $\chi_\cA(t)=t^3-10t^2+30t-21=(t-1)(t^2-9t+21)$, and for $K=\C$, $\Poin(M(\cA),t)=(1+t)(1+9t+21t^2)$;
\item $\cA$ is not free.
\end{enumerate}
\end{proposition}
\begin{proof}
(a) If $H_{ij}=H_{kl}$ with $\{i,j\}\ne\{k,l\}$, three of the $P$'s would lie in a plane. (b) Two planes sharing an index meet in $KP_i$, since $P_i,P_j,P_k$ are independent. $H_{ab}\supset KP_i$ forces $i\in\{a,b\}$. Let $\ell=H_{ij}\cap H_{kl}$ with $\{i,j\}\cap\{k,l\}=\emptyset$. If $\ell\subset H_{ab}$ for a third pair, then $\{a,b\}$ meets $\{i,j\}\cup\{k,l\}$ (only one index is left over), say $a=i$ and $b\ne j$. Then $\ell\subset H_{ij}\cap H_{ib}=KP_i$, so $P_i\in H_{kl}$, a contradiction. The same argument shows that $\ell$ is not one of the lines $KP_s$, and that the fifteen lines are distinct. (c) Hence $\sum_{\rk X=2}\mu(X)=5\cdot3+15\cdot1=30$, and $\chi_\cA(1)=0$ determines the constant term. (d) If $\cA$ were free, Theorem~\ref{thm:terao} would force the roots of $t^2-9t+21$, whose discriminant is $-3$, to be integers.
\end{proof}

\begin{proposition}[$3\times6$ case]\label{prop:36}
The projection $\pi\colon U_{3,6}\to U_{3,5}$ forgetting $v_6$ is a locally trivial fibration with fibre $M(\cA(v_1,\dots,v_5))$, and
\[\Poin(U_{3,6},t)=(1+t)^6(1+t^3)(1+t^5)(1+2t)(1+3t)(1+9t+21t^2).\]
Equivalently, $\Poin(X(3,6),t)=(1+2t)(1+3t)(1+9t+21t^2)$.
\end{proposition}
\begin{proof}
For $b=(v_1,\dots,v_5)\in U_{3,5}$ and $x\in\C^3$ put $\alpha_{ij}(b)(x)=[v_i,v_j,x]$, $1\le i<j\le5$. The maximal minors of $(v_1,\dots,v_5,x)$ are those of $b$ together with the $\Delta_{ij6}=\alpha_{ij}(b)(x)$. Hence
\[U_{3,6}=\{(b,x)\in U_{3,5}\times\C^3\mid \alpha_{ij}(b)(x)\ne0\ \text{for all } i<j\le5\},\qquad \pi=\mathrm{pr}_1.\]
Since $\ker\alpha_{ij}(b)=\operatorname{span}(v_i,v_j)=H_{ij}$, the vector $v_6$ must avoid exactly the planes $\operatorname{span}(v_i,v_j)$, $i<j\le5$. So $\pi^{-1}(v_1,\dots,v_5)=M(\cA(v_1,\dots,v_5))$.

\emph{Local triviality.} We apply Lemma~\ref{lem:isotopy} with $B=U_{3,5}$, $\ell=3$ and the ten forms $\alpha_{ij}$. These are polynomial in $b$, and they are nonzero because $v_i$ and $v_j$ are independent. It remains to show that the rank function
\[r(S)=\operatorname{codim}\textstyle\bigcap_{\{i,j\}\in S}H_{ij},\qquad S\subset\binom{[5]}{2},\]
does not depend on $b$. Any three of $v_1,\dots,v_5$ are linearly independent, so Proposition~\ref{prop:tenplanes} applies with $P_k=v_k$. By part (a), $r(S)=|S|$ if $|S|\le2$. Now let $|S|\ge3$. Then $r(S)\in\{2,3\}$, and $r(S)=2$ if and only if the planes in $S$ contain a common line. This line is then a rank-two flat lying on at least three planes of $\cA$. By part (b), this happens if and only if all pairs in $S$ share a common index $i$, the line being $\C v_i$. Thus $r(S)$ depends only on $S$. Lemma~\ref{lem:isotopy} now shows that $\pi$ is a locally trivial fibration with fibre $M(\cA(v_1,\dots,v_5))$.

\emph{Cohomology.} The classes $\omega_{ij}=[d\log\Delta_{ij6}]\in H^1(U_{3,6};\C)$ restrict on every fibre to the Orlik--Solomon generators $[d\log\alpha_{H_{ij}}]$. Since the Orlik--Solomon algebra is generated in degree one, we can choose monomials $c_1,\dots,c_s$ in the $\omega_{ij}$ whose restrictions form a basis of $H^*(\pi^{-1}(b_0);\C)$ for one fixed $b_0$. Consider the set of $b$ for which the restrictions of $c_1,\dots,c_s$ form a basis of $H^*(\pi^{-1}(b);\C)$. This set is open and closed. Indeed, let $V$ be a ball over which $\pi$ is trivial. Then the restriction maps $H^*(\pi^{-1}(V))\to H^*(\pi^{-1}(b'))$ are isomorphisms for all $b'\in V$. As $U_{3,5}$ is connected, the $c_k$ restrict to a basis on every fibre. The Leray--Hirsch theorem then gives $\Poin(U_{3,6})=\Poin(U_{3,5})\cdot\Poin(M(\cA))$. Now use Corollary~\ref{cor:35top} and Proposition~\ref{prop:tenplanes}(c). The formula for $X(3,6)$ follows from Theorem~\ref{thm:complement} and Example~\ref{ex:Y36}.
\end{proof}

\begin{remark}
As a check, the coefficient of $t$ is $6+2+3+9=20=\binom63$, the number of components of $D_{3,6}$. Proposition~\ref{prop:36} shows that the Poincar\'e polynomial of $U_{m,n}$ does not in general factor as $\Poin(\GL_m)$ times linear factors over $\Q$. In contrast, for $(n-2)\times n$ and $2\times n$ it does, regardless of freeness.
\end{remark}

\section{Local structure of \texorpdfstring{$D_{m,n}$}{Dmn}}\label{sec:local}

\subsection{Points of full rank: local Gale duality}

\begin{proposition}[local Gale duality]\label{prop:localGale}
Let $p\in\C^{m\times n}$ have rank $m$, and choose $I$ with $\Delta_I(p)\ne0$ -- we may assume that after permuting columns $I=[m]$. Write $p=\Phi(g_0,W_0)$ as in Proposition~\ref{prop:slice}, and let $q=[I_{n-m}\,|\,W_0^T]\in\C^{(n-m)\times n}$, a point of rank $n-m$. Then there are isomorphisms of germs
\[(\C^{m\times n},D_{m,n},p)\cong(\C^{m^2},\emptyset,g_0)\times(\C^{m(n-m)},E_{m,n},W_0),\]
\[(\C^{(n-m)\times n},D_{n-m,n},q)\cong(\C^{(n-m)^2},\emptyset,\mathrm{id})\times(\C^{m(n-m)},E_{n-m,n},W_0^T),\]
and $(E_{m,n},W_0)\cong(E_{n-m,n},W_0^T)$ via transposition. In particular $D_{m,n}$ is free at $p$ if and only if $D_{n-m,n}$ is free at $q$. Consequently, $D_{m,n}$ is free at all points of rank $m$ if and only if $D_{n-m,n}$ is free at all points of rank $n-m$.
\end{proposition}
\begin{proof}
Apply Proposition~\ref{prop:slice} to both $D_{m,n}$ and $D_{n-m,n}$, then Proposition~\ref{prop:gale}(a) and Lemma~\ref{lem:standard}(d). Every $W_0\in\C^{m\times(n-m)}$ arises from some point of rank $m$, and every point of rank $n-m$ of $\C^{(n-m)\times n}$ is, after permuting columns, of the form $\Phi(g,W_0^T)$.
\end{proof}

\begin{corollary}\label{cor:fullrankfree}
\begin{enumerate}[label=\textup{(\alph*)},leftmargin=2em]
\item $D_{n-1,n}$ is free at all points of rank $n-1$, being locally a normal crossing divisor there \textup{(}Gale dual to $D_{1,n}$\textup{)}.
\item $D_{n-2,n}$ is free at all points of rank $n-2$, because $D_{2,n}$ is free \cite[Theorem~3.3]{Yim} -- see also Appendix~\ref{app:2n}. In particular, $D_{3,5}$ is free at every matrix of rank $3$, and $D_{4,6}$ at every matrix of rank $4$.
\end{enumerate}
\end{corollary}

So for $(n-2)\times n$ matrices, any failure of freeness is concentrated on the determinantal variety of matrices of rank $\le n-3$.

\subsection{Points of rank $m-1$: a chart}

\begin{proposition}\label{prop:chart}
Let $\Omega'\subset\C^{m\times n}$ be the open set where the leading $(m-1)\times(m-1)$ minor $\det(x_{rk})_{r,k<m}$ is nonzero. Every $M\in\Omega'$ can be written uniquely as
\[M=\begin{pmatrix}B&0\\ r&1\end{pmatrix}\cdot N(Z),\qquad N(Z)=\begin{pmatrix}I_{m-1}&X\\0&z\end{pmatrix},\qquad Z=\begin{pmatrix}X\\ z\end{pmatrix}\in\C^{m\times(n-m+1)},\]
with $B\in\GL_{m-1}$ and $r\in\C^{1\times(m-1)}$. This gives an isomorphism $\Omega'\cong\GL_{m-1}\times\C^{m-1}\times\C^{m\times(n-m+1)}$. Moreover, $\Delta_I(M)=\det(B)\,\Delta_I(N(Z))$, and the maximal minors of $N(Z)$ are, up to sign, exactly the minors of $Z$ \textup{(}of all sizes $\ge1$\textup{)} whose row set contains the last row $m$, each occurring once. Hence
\[D_{m,n}\cap\Omega'\cong\GL_{m-1}\times\C^{m-1}\times E'_{m,n},\qquad E'_{m,n}=V(f'_{m,n}),\quad f'_{m,n}=\prod_{\substack{P\ni m,\ |P|=|Q|}}\det Z_{P,Q}.\]
A point of $\Omega'$ has rank $m-1$ if and only if $z=0$. Every point of $\C^{m\times n}$ of rank $\ge m-1$ can be moved into $\Omega'$ by an element of $S_n\times\GL_m$, which preserves $D_{m,n}$.
\end{proposition}
\begin{proof}
Write $M=\binom{A}{y}$ with $A\in\C^{(m-1)\times n}$ and $B=A_{[m-1]}$. Then
\[\begin{pmatrix}B&0\\ r&1\end{pmatrix}^{-1}M=\begin{pmatrix}B^{-1}A\\ y-rB^{-1}A\end{pmatrix},\]
and the first $m-1$ entries of the last row vanish exactly when $r=y_{[m-1]}$. This proves existence and uniqueness, and the inverse map is visibly algebraic. For the minors, expand along the identity columns as in Lemma~\ref{lem:minorsIW}: for $I=I_1\sqcup((m-1)+I_2)$ one gets $\pm\det Z_{([m]\setminus I_1),I_2}$, and $[m]\setminus I_1\ni m$. Conversely each pair $(P,Q)$ with $m\in P$, $|P|=|Q|\ge1$ arises from exactly one $I$. The rank statement holds because $\rk N(Z)=m-1+\rk z$. For the last claim, move $m-1$ independent columns to the front and then apply a row operation.
\end{proof}

\begin{example}[the $3\times5$ chart]\label{ex:chart35}
For $m=3$, $n=5$ we have $Z=\begin{pmatrix}a\\ b\\ c\end{pmatrix}\in\C^{3\times3}$, rows $a,b,c$ indexed by the columns $3,4,5$, and
\[f'_{3,5}=\prod_{k}c_k\cdot\prod_{j<k}(a_jc_k-a_kc_j)(b_jc_k-b_kc_j)\cdot\det Z,\]
a product of ten forms with $\deg f'_{3,5}=18$ in the $9$ variables $a,b,c$. The point $M_0=(e_1,e_2,0,0,0)$ corresponds to $Z=0$. The seven linear logarithmic fields of $f'_{3,5}$ are spanned by the row and column scalings ($3+3-1=5$ of them) and the two shears $a\mapsto a+\lambda c$ and $b\mapsto b+\lambda c$.
\end{example}

\subsection{Transversal arrangements at rank-two points of $D_{3,5}$}

\begin{proposition}\label{prop:transversal}
In the chart of Example~\ref{ex:chart35}, fix $X=\binom ab$. Then $f'_{3,5}(X,c)$ is the product of the ten linear forms in $c=(c_3,c_4,c_5)$ displayed in Example~\ref{ex:chart35}, and if $X\in Y_{2,5}$ these forms are nonzero and
\[f'_{3,5}(X,c)=\pm\,Q_{\cA(P_1,\dots,P_5)}(c),\qquad P_1=a,\ P_2=b,\ P_k=e_k\ (k=3,4,5),\]
where the minor $\Delta_I$ corresponds to the plane $H_{jk}=\operatorname{span}(P_j,P_k)$ with $\{j,k\}=[5]\setminus I$. The vectors $P_1,\dots,P_5$ are in general position if and only if the five columns $e_1,e_2,X_3,X_4,X_5$ of $[I_2\,|\,X]$ are pairwise independent, i.e.\ if and only if $X\in Y_{2,5}$. In that case $(P_1,\dots,P_5)$ is \textup{(}up to signs\textup{)} the Gale transform of the rank-two configuration $(e_1,e_2,X_3,X_4,X_5)$.
\end{proposition}
\begin{proof}
The plane $c_k=0$ contains $e_i,e_j$ with $\{i,j,k\}=\{3,4,5\}$, and corresponds to $\Delta_{12k}$, whose complement is $\{i,j\}$. The plane $b_jc_k-b_kc_j=0$ contains $b=P_2$ and $e_i$, and corresponds to $\Delta_{1jk}$, whose complement is $\{2,i\}$. The plane $a_jc_k-a_kc_j=0$ corresponds in the same way to $\Delta_{2jk}$ and $\{1,i\}$. Finally $\det Z=0$ is the plane $\operatorname{span}(a,b)$, which corresponds to $\Delta_{345}$. General position of the $P$'s means that $a_k,b_k\ne0$ and that all $2\times2$ minors of $X$ are nonzero, i.e.\ $X\in Y_{2,5}$. The Gale transform of $[I_2|X]$ is given by the columns of $[-X^T|I_3]$, see Proposition~\ref{prop:gale}(c).
\end{proof}

Thus, near a generic rank-two matrix, $D_{3,5}$ is, up to a smooth factor, the total space of a family, parametrized by $X$ near $X_0\in Y_{2,5}$, of copies of the non-free arrangement of Proposition~\ref{prop:tenplanes}. Using the torus action, the base of the family can be cut down to the two-dimensional moduli space $\Mo{}_{,5}$ of the configuration. The family is not analytically locally trivial, since the moduli of the five points vary (it is topologically locally trivial, cf.\ the proof of Proposition~\ref{prop:36}).

\begin{remark}\label{rem:family}
Non-freeness of the fibres does not formally imply non-freeness of the total space of a non-trivial family. We do not know a general statement of this kind. For $D_{3,5}$ non-freeness is established by computation in Section~\ref{sec:35}. As supporting evidence for a general principle, we checked the family of generic arrangements of five planes in $\C^3$ over its two-dimensional moduli space, $c_1c_2c_3(c_1+c_2+c_3)(c_1+sc_2+uc_3)$ in $\C^5$ with coordinates $(s,u,c)$: it is likewise non-free exactly along $\{c=0\}$.
\end{remark}

\section{The logarithmic vector fields of \texorpdfstring{$D_{3,5}$}{D3,5}}\label{sec:35}

In this section $m=3$, $n=5$, $R=\C[x_{rk}]$ has $15$ variables, $f=f_{3,5}$ has degree $30$, and $\cL=\cL_{3,5}$ is the $13$-dimensional space of linear logarithmic fields, see Lemma~\ref{lem:linear}.

\begin{definition}\label{def:Phi}
For $i\ne k$ in $[5]$ put $\widetilde\psi_{i\to k}=(-1)^{i+k}\sgn(i-k)\,\psi_{i\to k}$, with $\psi_{i\to k}$ as in \eqref{eq:psi}; here $c_{i\to k}$ is the product of the three minors $\Delta_I$ with $i\in I\not\ni k$. For each $k$ choose some $i(k)\ne k$ and put $\Phi_k=\widetilde\psi_{i(k)\to k}$. By Theorem~\ref{thm:35}(1) below, the class of $\Phi_k$ modulo $R\cL$ does not depend on the choice.
\end{definition}

\begin{theorem}\label{thm:35}
Let $\lambda_1,\dots,\lambda_{13}$ be a basis of $\cL$.
\begin{enumerate}[label=\textup{(\arabic*)},leftmargin=2em]
\item For each $k$, the four fields $\widetilde\psi_{i\to k}$, $i\ne k$, are congruent modulo $R\cL$, and none of them lies in $R\cL$.
\item $\lambda_1,\dots,\lambda_{13},\Phi_1,\dots,\Phi_5$ is a minimal system of homogeneous generators of $\Der(-\log D_{3,5})$. In particular, the degree-one part of $\Der(-\log D_{3,5})$ is $\cL$, and there are no minimal generators in degrees $2,\dots,9$.
\item There are $a_{rj}\in R_{10}$ such that
\begin{equation}\label{eq:syz}
\sum_{k=1}^5x_{rk}\,\Phi_k=\sum_{j=1}^{13}a_{rj}\,\lambda_j\qquad(r=1,2,3),
\end{equation}
and these three relations generate the module of syzygies. Hence the minimal graded free resolution is
\[0\longrightarrow R(-11)^3\xrightarrow{\ \left(\begin{smallmatrix}-A\\ M^T\end{smallmatrix}\right)\ }R(-1)^{13}\oplus R(-10)^5\longrightarrow\Der(-\log D_{3,5})\longrightarrow0,\qquad A=(a_{rj})^T.\]
\item $R\cL\cong R(-1)^{13}$ is a free submodule, and $\Der(-\log D_{3,5})/R\cL\cong\coker\big(M^T\colon R(-11)^3\to R(-10)^5\big)$.
\item For $p\in\C^{3\times5}$, the stalk $\Der(-\log D_{3,5})_p$ is minimally generated by exactly $18-\rk M(p)$ elements. The non-free locus of $D_{3,5}$ is $\{\rk M\le2\}$, which is irreducible of codimension $3$, and $\pd_R\Der(-\log D_{3,5})=1$.
\item In the $\Z^3\times\Z^5$-grading of Lemma~\ref{lem:psi}, $\Phi_k$ has degree $\big((3,3,3),(2,2,2,2,2)-e_k\big)$, and the syzygy \eqref{eq:syz} has degree $\big((3,3,3)+e_r,(2,2,2,2,2)\big)$.
\end{enumerate}
\end{theorem}

\begin{corollary}\label{cor:notfree}
$D_{3,5}$ is not a free divisor. This answers \cite[Question~3.4]{Yim} negatively for $3\times5$ matrices.
\end{corollary}

\subsection{Simplification}
The non-freeness of $D_{3,5}$ and the description of its non-free locus do not need the $15$-variable computation. It suffices to work in the chart of Example~\ref{ex:chart35}.

\begin{proposition}\label{prop:certificate}
Let $S=\C[a_3,a_4,a_5,b_3,b_4,b_5,c_3,c_4,c_5]$ and let $f'=f'_{3,5}\in S$ be as in Example~\ref{ex:chart35}. Then:
\begin{enumerate}[label=\textup{(\alph*)},leftmargin=2em]
\item $\Der(-\log f')$ has minimal free resolution $0\to S(-7)\to S(-1)^7\oplus S(-6)^3\to\Der(-\log f')\to0$, with the seven linear generators listed in Example~\ref{ex:chart35};
\item the non-free locus of $E'_{3,5}=V(f')$ is the linear subspace $\{c=0\}$.
\end{enumerate}
\end{proposition}
\begin{proof}
This is a Gr\"obner basis computation over $\Q$ (see Appendix~\ref{app:code}). By Lemma~\ref{lem:standard}(f), it gives the result over $\C$. For (b), the Fitting ideal $\mathrm{Fitt}_9$ is generated by the entries of the $10\times1$ presentation matrix, and its radical is $(c_3,c_4,c_5)$, then apply Lemma~\ref{lem:standard}(e).
\end{proof}

\begin{corollary}\label{cor:nonfreelocus}
The non-free locus of $D_{3,5}$ is exactly the determinantal variety $\{\rk M\le2\}$.
\end{corollary}
\begin{proof}
By Corollary~\ref{cor:fullrankfree}, $D_{3,5}$ is free at points of rank $3$. Let $p$ have rank $2$. By Proposition~\ref{prop:chart}, some $\sigma\in S_5\times\GL_3$ moves $p$ into $\Omega'$, where it corresponds to a point with $c=0$. By Proposition~\ref{prop:chart} and Lemma~\ref{lem:standard}(d), $D_{3,5}$ is free at $p$ if and only if $E'_{3,5}$ is free at the corresponding point, and it is not by Proposition~\ref{prop:certificate}(b). Points of rank $\le1$ lie in the closure of the rank-two locus, and the non-free locus is closed.
\end{proof}

\subsection{Proof of Theorem~\ref{thm:35}}
The fields $\lambda_j$ and $\psi_{i\to k}$ are logarithmic by Lemmas~\ref{lem:linear} and~\ref{lem:psi}. Statements (1), (2), (3) and (6) are verified by an exact computation in Macaulay2 (again, see Appendix~\ref{app:code}). We compute $\Der(-\log f)$ as the kernel of $R^{15}\oplus R^{10}\to R^{10}$, $(\theta,h)\mapsto(\theta(\Delta_I)-h_I\Delta_I)_I$, projected to $R^{15}$, and obtain the minimal free resolution of (3). We then check:
\begin{itemize}[leftmargin=2em]
\item membership of each $\psi_{i\to k}$ in $\Der(-\log f)$;
\item non-membership in $R\cL$;
\item the congruences in (1);
\item equality of the submodule generated by $\cL$ and $\Phi_1,\dots,\Phi_5$ with $\Der(-\log f)$, for several choices of $i(k)$;
\item the relations \eqref{eq:syz}, and that the syzygies of $(\lambda_1,\dots,\lambda_{13},\Phi_1,\dots,\Phi_5)$ are generated by three elements of degree $11$.
\end{itemize}
The multidegrees in (6) come from a separate computation in the $\Z^8$-graded ring.

For item (4), by Lemma~\ref{lem:linear} the $\lambda_j$ are $R$-independent, so $R\cL\cong R(-1)^{13}$. The quotient $\Der/R\cL$ is the cokernel of the projection of the syzygy module to the summand $R(-10)^5$. By (3) this projection is the image of $M^T$.

For item (5), by Lemma~\ref{lem:standard}(e), $\mu(\Der_p)=18-\rk P(p)$, where $P=\left(\begin{smallmatrix}-A\\M^T\end{smallmatrix}\right)$. Since $M(p)^T$ is a submatrix of $P(p)$, we get $\rk P(p)\ge\rk M(p)$, hence $\mu(\Der_p)\le18-\rk M(p)$. We show equality rank by rank.
\begin{itemize}[leftmargin=2em]
\item Rank $3$: $\mu\ge15$ always, since $\Der$ has rank $15$.
\item Rank $2$: $\mu\ge16$ by Corollary~\ref{cor:nonfreelocus}.
\item Rank $1$: substituting $M=uw^T$ with generic $u\in\C^3$, $w\in\C^5$ into a minimal presentation matrix gives a matrix of rank $1$ over $\C(u,w)$ (Appendix~\ref{app:code}). So all $2\times2$ minors of $P$ vanish on the rank-one locus, and $\mu\ge17$ there. The rank of $P(p)$ does not depend on the presentation, since $18-\rk P(p)=\mu(\Der_p)$ for any presentation on $18$ generators.
\item Rank $0$: all entries of $P$ have positive degree, so $\mu=18$.
\end{itemize}
The remaining assertions follow from item (3) and Corollary~\ref{cor:nonfreelocus}. \qed

\subsection{A bunch of remarks}

\begin{remark}
If $D_{3,5}$ were free, Lemma~\ref{lem:standard}(b) would force a basis consisting of the $13$ linear fields and two fields of degrees $d_1+d_2=30-13=17$. This is impossible by Theorem~\ref{thm:35}(2), since the next minimal generators occur in degree $10$. The resolution in (3) is consistent with the determinant of $\Der(-\log f)$ being $R(-\deg f)$: $13\cdot1+5\cdot10-3\cdot11=30$.
\end{remark}

\begin{remark}
Let $\rho\colon\{\rk M=3\}\to\Gr(3,5)$ send $M$ to its row space $\Lambda$. On the open set $\{\rk M=3\}$, the module $\coker(M^T)$ of Theorem~\ref{thm:35}(4) is the pull-back of the tautological quotient bundle $\C^5/\Lambda$, a vector bundle of rank $2$. So, over $\{\rk M=3\}$ and modulo the vector fields tangent to the $\GL_3\times(\C^*)^5$-orbits, the logarithmic vector fields of $D_{3,5}$ are sections of $\rho^*(\C^5/\Lambda)(-10)$. On all of $\C^{3\times5}$, the sheaf $\coker(M^T)$ is locally free exactly where $\rk M=3$.

In general, $\operatorname{rank}\Der(-\log D_{m,n})-\dim\cL_{m,n}=mn-m^2-n+1=(m-1)(n-m-1)$, the dimension of the moduli space of $n$ points in general position in $\PP^{m-1}$. For $2\times5$ this quotient is free of rank $2$ ($\Der(-\log D_{2,5})\cong R(-1)^8\oplus R(-6)^2$); for $3\times5$ it is the non-free rank-two sheaf above.
\end{remark}

\begin{remark}
All five $\Phi_k$ have row degree $(3,3,3)$. Hence the degree-$10$ part of $\Der/R\cL$ is a sum of copies of the character $\det^3$ of $\GL_3$, and its generators can be chosen $\SL_3$-equivariant, as they are. By the first fundamental theorem of invariant theory, every $\SL_3$-equivariant polynomial vector field on $\C^{3\times5}$ has the form $M\mapsto M\cdot C(M)$, with $C$ a $5\times5$ matrix of polynomials in the Pl\"ucker coordinates $\Delta_I$. For $\psi_{i\to k}$ one has $C=c_{i\to k}E_{ki}$, the simplest possible choice in the weight spaces of Theorem~\ref{thm:35}(6).
\end{remark}

\begin{remark}
Under Gale duality the Pl\"ucker coordinate $\Delta_I$ of $\Gr(3,5)$ corresponds, up to sign, to the coordinate $\Delta'_{[5]\setminus I}$ of $\Gr(2,5)$. The coefficient $c_{i\to k}=\prod_{I\ni i,\,I\not\ni k}\Delta_I$ then corresponds to $\prod_{J\ni k,\,J\not\ni i}\Delta'_J=c'_{k\to i}$, the coefficient of the field $\psi'_{k\to i}$ of $D_{2,5}$. The fields $\psi'_{k\to i}$ are logarithmic of degree $7$ (Lemma~\ref{lem:psi}), and they are not minimal generators of the free module $\Der(-\log D_{2,5})$, whose non-linear generators have degree $6$. We do not know whether this correspondence has a conceptual explanation at the level of logarithmic vector fields.
\end{remark}

\section{Other sizes}\label{sec:other}

\subsection{Transversal arrangements in general}
The chart of Proposition~\ref{prop:chart} gives a uniform description of the transversal structure of $D_{m,n}$ at generic points of rank $m-1$.

\begin{proposition}\label{prop:Hk}
In the chart of Proposition~\ref{prop:chart}, fix $X\in Y_{m-1,n}\subset\C^{(m-1)\times(n-m+1)}$, i.e.\ assume that the columns of $[I_{m-1}|X]$ are in general position. Then $f'_{m,n}(X,z)$ is a product of $\binom nm$ nonzero linear forms in $z\in\C^{n-m+1}$. They define the arrangement $\mathcal H_{n-m}(P)$ of the hyperplanes spanned by the $(n-m)$-element subsets of the $n$ vectors
\[P=(X_1,\dots,X_{m-1},e_1,\dots,e_{n-m+1})\]
in $\C^{n-m+1}$, where $X_p$ are the rows of $X$. Up to signs, $P$ is the Gale transform of the $n$ columns of $[I_{m-1}|X]$.
\end{proposition}
\begin{proof}
The minor $\det Z_{P',Q}$ with $m\in P'$ is linear in $z$. It vanishes at $z=e_j$ for $j\notin Q$ (the last row of $Z_{P',Q}$ is then zero) and at $z=X_p$ for $p\in P'\setminus\{m\}$ (two equal rows). These $(n-m+1-|Q|)+(|P'|-1)=n-m$ vectors are independent, since $P$ is in general position (Gale duality preserves general position), so they span the zero set. Conversely every $(n-m)$-subset of $P$ arises from exactly one pair $(P',Q)$.
\end{proof}

\begin{example}
\begin{enumerate}[label=(\alph*),leftmargin=2em]
\item $m=n-1$: $\mathcal H_1(P)$ consists of $n$ lines through the origin in $\C^2$, which is always free. This is consistent with $D_{n-1,n}$ being a linear free divisor.
\item $m=2$: $\mathcal H_{n-2}(P)$, for $P=(X_1,e_1,\dots,e_{n-1})$ in $\C^{n-1}$, is linearly equivalent (take $X_1=(1,\dots,1)$) to the essential braid arrangement $\{z_i=0\}\cup\{z_i=z_j\}$ in $\C^{n-1}$, of type $A_{n-1}$, which is free with exponents $1,\dots,n-1$. This is consistent with Yim's theorem.
\item $m=n-2$: $\mathcal H_2(P)$ is the arrangement of the $\binom n2$ planes spanned by pairs of $n$ generic vectors in $\C^3$ (Proposition~\ref{prop:pairs}).
\item $3\times6$: $\mathcal H_3(P)$ consists of $20$ hyperplanes in $\C^4$ spanned by triples of $6$ generic vectors. A direct computation of the intersection lattice for a random integral configuration gives
\[\chi(t)=(t-1)(t^3-19t^2+126t-300),\]
whose cubic factor is monotone and has a single real root in $(7,8)$. So this arrangement is not free, by Theorem~\ref{thm:terao}.
\end{enumerate}
\end{example}

\begin{proposition}\label{prop:pairs}
Let $n\ge4$ and let $P_1,\dots,P_n\in\C^3$ be generic. Then the arrangement of the $N=\binom n2$ planes $\operatorname{span}(P_i,P_j)$ has characteristic polynomial
\[\chi(t)=(t-1)\big(t^2-(N-1)t+b_n\big),\qquad b_n=n(n-2)+3\tbinom n4-N+1.\]
The discriminant of the quadratic factor is $-\tfrac14(n-2)(n-3)(n^2-5n+2)$, which is negative for $n\ge5$. Hence the arrangement is free for $n=4$ (it is the braid arrangement, $\chi=(t-1)(t-2)(t-3)$) and not free for $n\ge5$.
\end{proposition}
\begin{proof}
For generic points, the rank-two flats are the $n$ lines $\C P_i$, each in $n-1$ planes, and the $\frac12\binom n2\binom{n-2}2=3\binom n4$ lines $\operatorname{span}(P_i,P_j)\cap\operatorname{span}(P_k,P_l)$ with $\{i,j\}\cap\{k,l\}=\emptyset$, each in exactly two planes; compare Proposition~\ref{prop:tenplanes}, where for $n=5$ general position suffices. Therefore $\sum_{\rk X=2}\mu(X)=n(n-2)+3\binom n4$, and $\chi(1)=0$ gives the constant term. The discriminant is a direct computation. For $n=5,\dots,8$ the lattice was also computed for random integral configurations; the quadratic factors are $t^2-9t+21$, $t^2-14t+55$, $t^2-20t+120$ and $t^2-27t+231$.
\end{proof}

\subsection{Computations}
Table~\ref{tab:sizes} summarizes what we know. The entries ``not free'' for $3\times6$, $4\times6$ and $5\times7$ come from computing $\Der(-\log f'_{m,n})$ in the chart of Proposition~\ref{prop:chart} at $Z=0$, i.e.\ at the point $(e_1,\dots,e_{m-1},0,\dots,0)$: the number of minimal generators exceeds the number of variables, so the chart divisor is not free at the origin. By Proposition~\ref{prop:chart} and Lemma~\ref{lem:standard}(d), $D_{m,n}$ is then not free.

\begin{table}[ht]
\centering
\renewcommand{\arraystretch}{1.15}
\begin{tabular}{@{}llll@{}}
\toprule
size & freeness & degrees of minimal generators & source\\
\midrule
$1\times n$ & yes & $(1^n)$ & normal crossings\\
$m\times(m+1)$ & yes & $(1^{m(m+1)})$ & \cite{BM}\\
$2\times n$ & yes & $(1^{n+3},(n+1)^{n-3})$ & \cite{Yim}\\
$3\times5$ & no & $(1^{13},10^5)$; syzygies $R(-11)^3$ & Theorem~\ref{thm:35}\\
$4\times6$ & no & chart ($12$ vars): $(1^9,8^6)$ & Appendix~\ref{app:code}, Part 7\\
$5\times7$ & no & chart ($15$ vars): $(1^{11},10^{10})$ & Appendix~\ref{app:code}, Part 7\\
$3\times6$ & no & chart ($12$ vars): $(1^8,11^8,12^{12},13^{19},17)$ & Appendix~\ref{app:code}, Part 7\\
\bottomrule
\end{tabular}
\medskip
\caption{Freeness of $D_{m,n}$. For $4\times6$, $5\times7$ and $3\times6$ the table lists the minimal generators of $\Der(-\log f'_{m,n})$ in the chart at a rank-$(m-1)$ point.}\label{tab:sizes}
\end{table}

Note the pattern for the $(n-2)\times n$ charts, $n=5,6,7$: there are $2n-3$ linear generators and $\binom{n-2}2$ generators of degree $2(n-2)$. For $4\times6$ we also evaluated a presentation matrix of $\Der(-\log f'_{4,6})$ at randomly chosen rational points: at points with $z=0$ (rank $3$) the stalk needs $15$ generators, while at points with $z\ne0$ (rank $4$) it is free, as it must be by Corollary~\ref{cor:fullrankfree}. 

Finishing this paper, before presenting computational scripts, we would like to formulate an updated version of Yim's question about determinantal arrangements.
\begin{conjecture}\label{conj}
For $1\le m<n$, the divisor $D_{m,n}$ is free if and only if $m\le2$ or $n=m+1$. Moreover, for $3\le m\le n-2$ the non-free locus contains all matrices of rank at most $m-1$; for $m=n-2$ it is exactly $\{\rk M\le n-3\}$.
\end{conjecture}

For $m=n-2$, freeness at matrices of full rank is Corollary~\ref{cor:fullrankfree}. By Gale duality, the case $m>n/2$ at full-rank points reduces to $n-m<n/2$.

\appendix
\section{Computations}\label{app:code}

All Macaulay2 computations reported in this paper were performed with
Macaulay2 \cite{M2}, version 1.22, over $\Q$. All calculations are exact.
Since the defining polynomials have rational coefficients, 
Lemma~\ref{lem:standard}(f) shows that the resulting statements about
logarithmic derivation modules remain valid after extension of scalars to
$\C$.

The complete source file is
\texttt{maxminors\_computations.m2}. It is supplied as ancillary material
and, for reproducibility and version control, we propose to maintain the
same file in a public GitHub repository
\url{https://github.com/Idfewcount/maximal-minors-arrangements}

%

The script is self-contained and can be executed from the command line by
\begin{center}\texttt{M2 script --maxminors\_computations.m2}.
\end{center}
No additional Macaulay2 packages are required.

\subsection*{Computation of logarithmic derivations}

We briefly describe the construction implemented in the script. Let
\[
R=\Q[y_1,\ldots,y_N],\qquad
f=\prod_{i=1}^{s}f_i,
\]
where the $f_i$ are the pairwise non-associate irreducible homogeneous
factors of $f$. Put
\[
J=\left(\frac{\partial f_i}{\partial y_j}\right)_
   {\substack{1\le i\le s\\1\le j\le N}},
\qquad
H=-\operatorname{diag}(f_1,\ldots,f_s).
\]
The kernel of the matrix
\[
(J\mid H)
\]
consists of pairs $(g,h)\in R^N\oplus R^s$ satisfying
\[
\sum_{j=1}^{N}g_j\frac{\partial f_i}{\partial y_j}
   =h_i f_i
   \qquad (1\le i\le s).
\]
Projecting this kernel onto the first $N$ coordinates therefore gives
\[
\Der(-\log f)
 =
\left\{
\sum_{j=1}^{N}g_j\frac{\partial}{\partial y_j}
\ \middle|\
\sum_{j=1}^{N}g_j\frac{\partial f_i}{\partial y_j}\in(f_i)
\text{ for every }i
\right\}.
\]
This is implemented by the function \texttt{logDer}. The graded shifts are
chosen so that all maps are homogeneous. The auxiliary function
\texttt{logDerUngraded} implements the same construction without a grading
and is used only for the family considered in Remark~\ref{rem:family}.

\subsection*{Organization of the script}

The file is divided into the following parts.

\begin{description}[leftmargin=6em,labelwidth=5.5em,style=sameline]

\item[Part 1]
Sanity checks in known free cases. The script computes
$\Der(-\log D_{2,n})$ for $3\le n\le7$ and
$\Der(-\log D_{3,4})$. The resulting degrees agree with the known free
cases recalled in the introduction and in Appendix~\ref{app:2n}.

\item[Part 2]
The nine-variable computation used in
Proposition~\ref{prop:certificate}. For
\[
N_9=
\begin{pmatrix}
1&0&a_3&a_4&a_5\\
0&1&b_3&b_4&b_5\\
0&0&c_3&c_4&c_5
\end{pmatrix},
\]
the script forms the ten nonzero $3\times3$ minors and computes the
logarithmic derivation module of their product. It also computes the
relevant Fitting ideal and its radical.

\item[Part 3]
The computation of $\Der(-\log D_{3,5})$ used in
Theorem~\ref{thm:35}(1)--(3). It determines the minimal generators and
the minimal graded free resolution, constructs the fields
$\psi_{i\to k}$ and $\widetilde\psi_{i\to k}$, verifies their
logarithmicity, checks the congruences modulo $R\cL$, and verifies that
$\cL$ together with any admissible choice of
$\Phi_1,\ldots,\Phi_5$ generates the complete logarithmic derivation
module. It also verifies the three relations \eqref{eq:syz}.

\item[Part 3b]
The multigraded computation for Theorem~\ref{thm:35}(6). The polynomial
ring is endowed with the $\Z^3\times\Z^5$ grading used in
Lemma~\ref{lem:psi}, and the multidegrees of the minimal generators and
syzygies are computed directly.

\item[Part 4]
The rank-one computation used in the proof of
Theorem~\ref{thm:35}(5). The minimal presentation matrix is restricted to
the generic rank-one matrix $M=uw^T$, and its rank is computed over the
fraction field $\Q(u_1,u_2,u_3,w_1,\ldots,w_5)$.

\item[Part 5]
An independent computation for the ten-plane arrangement
$\cA(P_1,\ldots,P_5)$ with
\[
P_1=e_1,\quad P_2=e_2,\quad P_3=e_3,\quad
P_4=(2,3,5),\quad P_5=(7,-1,4).
\]
It computes its logarithmic derivation module and its projective
dimension. This part serves as an algebraic consistency check for the
non-freeness established combinatorially in
Proposition~\ref{prop:tenplanes}.

\item[Part 6]
The computation mentioned in Remark~\ref{rem:family}. For the family
\[
c_1c_2c_3(c_1+c_2+c_3)(c_1+sc_2+uc_3)
\]
the script computes the logarithmic derivation module and the radical of
the Fitting ideal defining its non-free locus.

\item[Part 7]
The chart computations at rank-$(m-1)$ points for the sizes
$3\times5$, $4\times6$, $5\times7$, and $3\times6$. These calculations
produce the generator degrees recorded in Table~\ref{tab:sizes}.

\item[Part 8]
The local computation for the $4\times6$ chart discussed after
Table~\ref{tab:sizes}. The presentation matrix is evaluated at rational
points with $z=0$ and at a point with $z\ne0$. The random seed is fixed by
\texttt{setRandomSeed 2026}, so the calculation is reproducible. These
tests use exact rational arithmetic and provide auxiliary evidence only.

\end{description}

\subsection*{Computational certificates for the \texorpdfstring{$3\times5$}{3x5} case}

For the nine-variable chart of Proposition~\ref{prop:certificate}, the
script returns ten minimal generators, seven in degree $1$ and three in
degree $6$, together with one syzygy in degree $7$. Thus
\[
0\longrightarrow S(-7)
\longrightarrow S(-1)^7\oplus S(-6)^3
\longrightarrow\Der(-\log f'_{3,5})
\longrightarrow0.
\]
If $P_9$ denotes a minimal presentation matrix, the script also returns
\[
\sqrt{\operatorname{Fitt}_9\big(\Der(-\log f'_{3,5})\big)}
 =
\sqrt{I_1(P_9)}
 =
(c_3,c_4,c_5).
\]
This is the computational input in
Proposition~\ref{prop:certificate}. In particular, the non-free locus in
this chart is exactly the linear subspace $\{c=0\}$.

For the full $3\times5$ divisor, the computation gives $18$ minimal
generators, with degrees
\[
1^{13},\qquad 10^5,
\]
and three syzygies of degree $11$. Hence it produces the minimal
resolution
\[
0\longrightarrow R(-11)^3
\longrightarrow R(-1)^{13}\oplus R(-10)^5
\longrightarrow\Der(-\log D_{3,5})
\longrightarrow0.
\]
The script independently checks that every $\psi_{i\to k}$ is
logarithmic, that no $\psi_{i\to k}$ lies in $R\cL$, and that for fixed
$k$ the four fields $\widetilde\psi_{i\to k}$ are congruent modulo
$R\cL$. For several different choices of the representatives
$\Phi_k=\widetilde\psi_{i(k)\to k}$ it verifies
\[
R\cL+\sum_{k=1}^{5}R\Phi_k
   =\Der(-\log D_{3,5})
\]
and
\[
\sum_{k=1}^{5}x_{rk}\Phi_k\in R\cL,
\qquad r=1,2,3.
\]
The kernel of the resulting $18$-generator presentation is minimally
generated by three elements of degree $11$.

The separate multigraded calculation gives
\[
\deg\Phi_k
 =
\big((3,3,3),(2,2,2,2,2)-e_k\big)
\]
and
\[
\deg\!\left(\sum_{k=1}^{5}x_{rk}\Phi_k\right)
 =
\big((3,3,3)+e_r,(2,2,2,2,2)\big),
\qquad r=1,2,3,
\]
as stated in Theorem~\ref{thm:35}(6).

Finally, after substituting the generic rank-one matrix
\[
M=uw^T
\]
into a minimal presentation matrix of
$\Der(-\log D_{3,5})$, the script computes its rank over
$\Q(u,w)$ to be exactly $1$. This is the only additional computer
calculation needed in the rank-one case of
Theorem~\ref{thm:35}(5). The remaining steps in the proof of that item,
including the identification of the non-free locus, are the geometric
and module-theoretic arguments given in Section~\ref{sec:35}.
\section{A short proof that \texorpdfstring{$D_{2,n}$}{D2,n} is free}\label{app:2n}

Corollary~\ref{cor:fullrankfree}(b) uses Yim's theorem that $D_{2,n}$ is free. For completeness we give a short proof, which replaces the determinant computations of \cite[\S3]{Yim} by Saito's criterion in the form of Lemma~\ref{lem:standard}(b). Let $m=2$, $n\ge4$, and write $[a,b]=\det(a|b)$ for $a,b\in\C^2$.

\emph{The fields.} The $n+3$ linear fields: a basis of $\cL_{2,n}$ (Lemma~\ref{lem:linear}). For $u\in\C^2$ put
\[\varphi_u=\sum_{k=4}^n[v_2,v_k][v_3,v_k]\prod_{\substack{s\ge4\\ s\ne k}}[u,v_s]\;(v_1\cdot\dd{v_k}).\]
This is a binary form of degree $n-4$ in $u$. Its coefficients $\varphi_0,\dots,\varphi_{n-4}$ (for $u=(1,-T)$, $\varphi_j$ is the coefficient of $T^j$) are the fields of degree $n+1$ of \cite{Yim}.

\emph{They are logarithmic.} Write $\varphi_u=\sum_kc_k(v_1\cdot\dd{v_k})$. Then $\varphi_u(\Delta_{ij})=0$ for $i,j\le3$, and $\varphi_u(\Delta_{1k})=c_k[v_1,v_1]=0$. For $i\in\{2,3\}$, $\varphi_u(\Delta_{ik})=c_k[v_i,v_1]$ is divisible by $[v_i,v_k]$. For $4\le i<j$,
\begin{gather*}
\varphi_u(\Delta_{ij})=Q\cdot\big(H(v_i,v_j)-H(v_j,v_i)\big),\qquad\text{where}\\
H(a,b)=[v_2,a][v_3,a][v_1,b][u,b],\qquad Q=\prod_{s\ge4,\ s\ne i,j}[u,v_s].
\end{gather*}
Since $H$ is bihomogeneous of bidegree $(2,2)$, the bracket vanishes when $v_j=\lambda v_i$ or $v_i=0$, i.e.\ on $V(\Delta_{ij})$, so it is divisible by $\Delta_{ij}$.

\emph{Degrees.} $(n+3)\cdot1+(n-3)(n+1)=n(n-1)=\deg f_{2,n}$.

\emph{Independence at one point.} By Lemma~\ref{lem:standard}(b) it remains to find a point where the $2n$ fields are linearly independent.
\begin{itemize}[leftmargin=2em]
\item Let $\rho=(\rho_4,\dots,\rho_n)\colon U_{2,n}\to\C^{n-3}$, $\rho_k=[v_1,v_k][v_2,v_3]/([v_2,v_k][v_1,v_3])$, be the cross-ratio map. It is constant on $G$-orbits, so the linear fields lie in $\ker d\rho$.
\item At $p=(e_1,e_2,e_1+e_2,e_1+s_4e_2,\dots,e_1+s_ne_2)$ with distinct $s_k\notin\{0,1\}$ one computes
\[d\rho_k(\varphi_j)=s_k(s_k-1)\,\sigma_{n-4-j}(s_4,\dots,\widehat{s_k},\dots,s_n),\]
where $\sigma_r$ is the $r$-th elementary symmetric polynomial. The matrix $\big(\sigma_r(\widehat{s_k})\big)_{k,r}$ is invertible, because the polynomials $\prod_{l\ne k}(z-s_l)=\sum_r(-1)^r\sigma_r(\widehat{s_k})z^{n-4-r}$ are linearly independent (evaluate at $z=s_i$).
\item So $d\rho_p$ is surjective. Its kernel has dimension $n+3$, it contains the $(n+3)$-dimensional tangent space of $G\cdot p$, see Lemma~\ref{lem:linear}, and the linear fields span that tangent space.
\end{itemize}
Hence all $2n$ fields are independent at $p$, and $D_{2,n}$ is free with exponents $(1^{n+3},(n+1)^{n-3})$. \qed
\section*{Acknowledgements}
The author would like to thank Professor Piotr Pokora for stimulating discussions and valuable comments on the manuscript.



\begin{thebibliography}{GGMS87}

\bibitem[Arn69]{Arnold}
V.~I. Arnold, \emph{The cohomology ring of the colored braid group}, Mat. Zametki \textbf{5} (1969), 227--231.

\bibitem[BM06]{BM}
R.-O. Buchweitz and D. Mond, \emph{Linear free divisors and quiver representations}, in: Singularities and Computer Algebra, London Math. Soc. Lecture Note Ser. \textbf{324}, Cambridge Univ. Press, 2006, 41--77.

\bibitem[Dim92]{Dimca}
A. Dimca, \emph{Singularities and Topology of Hypersurfaces}, Universitext, Springer-Verlag, New York, 1992.

\bibitem[DO88]{DO}
I. Dolgachev and D. Ortland, \emph{Point sets in projective spaces and theta functions}, Ast\'erisque \textbf{165} (1988).

\bibitem[FN62]{FN}
E. Fadell and L. Neuwirth, \emph{Configuration spaces}, Math. Scand. \textbf{10} (1962), 111--118.

\bibitem[GGMS87]{GGMS}
I.~M. Gelfand, R.~M. Goresky, R.~D. MacPherson and V.~V. Serganova, \emph{Combinatorial geometries, convex polyhedra, and Schubert cells}, Adv. Math. \textbf{63} (1987), 301--316.

\bibitem[GMNS09]{GMNS}
M. Granger, D. Mond, A. Nieto-Reyes and M. Schulze, \emph{Linear free divisors and the global logarithmic comparison theorem}, Ann. Inst. Fourier \textbf{59} (2009), 811--850.

\bibitem[JW23]{JanaszWisniewska}
M.~Janasz and P.~Wi\'sniewska, \emph{A note on free determinantal hypersurface arrangements in $\PP^{14}_{\C}$}, Proyecciones Journal of Mathematics \textbf{42} (2023), no.~4, 1093--1104, \url{https://doi.org/10.22199/issn.0717-6279-5780}.

\bibitem[M2]{M2}
D.~R. Grayson and M.~E. Stillman, \emph{Macaulay2, a software system for research in algebraic geometry}, available at \url{https://macaulay2.com}.

\bibitem[OT92]{OT}
P. Orlik and H. Terao, \emph{Arrangements of Hyperplanes}, Grundlehren Math. Wiss. \textbf{300}, Springer, 1992.

\bibitem[Ran89]{Randell}
R. Randell, \emph{Lattice-isotopic arrangements are topologically isomorphic}, Proc. Amer. Math. Soc. \textbf{107} (1989), 555--559.

\bibitem[Sai80]{Saito}
K. Saito, \emph{Theory of logarithmic differential forms and logarithmic vector fields}, J. Fac. Sci. Univ. Tokyo Sect. IA Math. \textbf{27} (1980), 265--291.

\bibitem[Sta72]{Stanley}
R.~P. Stanley, \emph{Supersolvable lattices}, Algebra Universalis \textbf{2} (1972), 197--217.

\bibitem[Ter81]{Terao81}
H. Terao, \emph{Generalized exponents of a free arrangement of hyperplanes and Shepherd--Todd--Brieskorn formula}, Invent. Math. \textbf{63} (1981), 159--179.

\bibitem[Ter86]{Terao86}
H. Terao, \emph{Modular elements of lattices and topological fibration}, Adv. Math. \textbf{62} (1986), 135--154.

\bibitem[Yim15]{Yim}
A. Yim, \emph{Homological properties of determinantal arrangements}, arXiv:1506.05145v2.

\end{thebibliography}
\end{document}